\documentclass[12pt,reqno]{amsart}
\usepackage{amsmath,amssymb,enumerate}
\usepackage[dvipdfmx]{graphicx}
\usepackage{color}
\usepackage{comment,fancybox,comment}
\usepackage[all]{xy}
\usepackage{bm}

\newtheorem{tm}{Theorem}[section]
\newtheorem{lm}[tm]{Lemma}

\newtheorem{df}[tm]{Definition}
\newtheorem{pr}[tm]{Proposition}

\makeatletter
\newcommand{\subscripts}[3]{%
  \@mathmeasure\z@\displaystyle{#2}%
  \global\setbox\@ne\vbox to\ht\z@{}\dp\@ne\dp\z@
  \setbox\tw@\box\@ne
  \@mathmeasure4\displaystyle{\copy\tw@_{#1}}%
  \@mathmeasure6\displaystyle{{#2}_{#3}}%
  \dimen@-\wd6 \advance\dimen@\wd4 \advance\dimen@\wd\z@
  \hbox to\dimen@{}\mathop{\kern-\dimen@\box4\box6}%
}
\makeatother

\newcommand{\nn}{\nonumber}

\newcommand{\e}{\mathrm{e}}
\newcommand{\dd}{\mathrm{d}}

\newcommand{\ve}{\varepsilon}

\newcommand{\ii}{\mathrm{i}}

\newcommand{\al}{\alpha}
\newcommand{\bt}{\beta}

\newcommand{\Gm}{\Gamma}
\newcommand{\dl}{\delta}

\newcommand{\zt}{\zeta}

\newcommand{\sg}{\sigma}

\newcommand{\C}{\mathbb{C}}
\newcommand{\R}{\mathbb{R}}
\newcommand{\N}{\mathbb{N}}

\newcommand{\E}{\mathbb{E}}
\newcommand{\Prime}{\mathbb{P}}

\allowdisplaybreaks[3]

\makeatletter
 
 \@addtoreset{equation}{section}
\makeatother

\begin{document}
\title[Asymptotic behaviors of the Riemann 
zeta distribution]
{
Asymptotic behaviors of convolution powers of 
the Riemann zeta distribution
}
\author[T. Aoyama]{Takahiro Aoyama}
\author[R. Namba]{Ryuya Namba}
\author[K. Ota]{Koki Ota}
\date{\today}
\address[T. Aoyama]{Department of Applied Mathematics, 
Faculty of Science, 
Okayama University of Science, 
1-1 Ridaicho, Kita-ku, Okayama 700-0005, Japan}
\email{{\tt{aoyama@xmath.ous.ac.jp}}}
\address[R. Namba]{Department of Mathematics,
Faculty of Education,
Shizuoka University, 836, Ohya, Suruga-ku, Shizuoka, 422-8529, Japan}
\email{{\tt{namba.ryuya@shizuoka.ac.jp}}}
\address[K. Ota]{Agile Solution Department, 
Smart Solution Division, Resona Digital I Inc.,
2-13, Shinsenri-Nishimachi 1-chome, Toyonaka, Osaka, 560-0083, Japan}
\email{{\tt pd307dcy@s.okayama-u.ac.jp}}
\subjclass[2010]{Primary 60E; Secondary 11M}
\keywords{Riemann zeta distribution ; 
convolution power ; local limit theorem.}

\maketitle 

\begin{abstract}
In probability theory, there exist discrete and continuous distributions. 
Generally speaking, we do not have sufficient kinds and properties of discrete ones compared to the continuous ones.
In this paper, we treat the Riemann zeta distribution as a representative of few known discrete distributions with infinite supports.
Some asymptotic behaviors of convolution powers of 
the Riemann zeta distribution are discussed. 
\end{abstract}


\section{{\bf Introduction and main results}}
\label{Intro}

Let $\varphi$ be a complex-valued 
function whose support 
$\mathcal{S}=\mathcal{S}_1 \subset \mathbb{R}$ 
is at most countable. 
For $n \in \N$, the $n$-th convolution power of the function $\varphi$ 
is defined by $\varphi^{*1}:=\varphi$ and 
    \[
    \varphi^{*n}(x)
    :=\sum_{y \in \mathcal{S}}
    \varphi^{*(n-1)}(x-y)\varphi(y), \qquad n=2, 3, \dots, \, 
    x \in \mathcal{S}_n,
    \]
where $\mathcal{S}_n$ denotes the countable support of 
$\varphi^{*n}$ which may be different from $\mathcal{S}$ in general.  

It has been interesting problems to investigate some asymptotic  
behaviors of the convolution powers $\varphi^{*n}$ 
as $n \to \infty$
from not only analytic but probabilistic perspectives. 
Let us first look at the cases where $\mathcal{S}$ is finite. 
One of motivations behind such kinds of problems was found  
in the context of numerical difference schemes 
for partial differential equations. 
We refer to e.g.~Thom\'ee \cite{Thomee1, Thomee2} for details. 
Moreover, motivated by the de Forest's local limit theorems
and statistical data smoothing procedures, 
Greville \cite{Greville} and Schoenberg \cite{Schoen}
treated finitely supported functions 
$\varphi$ with values in $\R$ and established 
their local limit theorems. 
Namely, the uniform convergence of the convolution power 
$\varphi^{*n}(x)$ as $n \to \infty$ was obtained
and its leading term was shown to be an analytic function 
like heat kernels.  
We note that some well-known properties of 
Fourier transforms play 
a crucial role in obtaining such asymptotic behaviors. 
Afterwards,
Randles and Saloff-Coste \cite{RS1} extended 
the local limit theorems to the cases where  
the finitely supported function $\varphi$ is complex-valued. 
Furthermore, the corresponding leading term is shown to be
governed by a complex-valued analytic function which is  
regarded as an evaluation of a function similar to the 
heat kernel at some imaginary time. 
We also refer to Diaconis--Saloff-Coste \cite{DS} 
for related results on
such local limit theorems and 
Randles--Saloff-Coste \cite{RS2} 
for extensions of the results in \cite{RS1}
to multidimensional cases with concrete examples.

In view of probability theory, 
these results exactly help us to reveal 
the detailed asymptotic behaviors of random walks
whose one-step distribution consists of finitely many 
mass points. 
In spite of such developments, 
asymptotic behaviors of convolution powers of 
countably supported functions 
have not been investigated so much, 
though there exist a number of discrete
probability distributions whose support is countable. 
We emphasize that main difficulties in considering these problems 
can be often replaced by the complexities of convolution powers of functions with countable supports.

Therefore, our aim is to provide  
some asymptotic behaviors of convolution powers 
of treatable functions whose support is countable. 
Particularly, in the present paper, 
we try to focus on the 
so-called Riemann zeta distribution on $\R$, which is 
known as an important example of countably supported 
probability distribution on $\R$. 

In order to introduce this distribution, 
we start with the definition
of the Riemann zeta function $\zeta(s)$.  

\begin{df}[Riemann zeta function]
The Riemann zeta function $\zeta(s)$ 
is a function of a complex variable
$s = \sg + {\rm i}t$ for $\sg > 1$ 
and $t \in \R$ defined by   
    \begin{equation}
    \label{zt}
    \zt(s) := \sum^{\infty}_{m=1}\frac{1}{m^s} 
    = \prod_{p \in \Prime} \Big(1-\frac{1}{p^s} \Big)^{-1},
    \end{equation}
where we denote by $\Prime$ the set of all prime numbers.
The product representation of {\rm (\ref{zt})} 
is called the Euler product.
\end{df}

It is known that the function $\zeta(s)$ 
converges absolutely in the half-plane 
$\{\sigma+\ii t \, | \, \sigma>1\}$ 
and uniformly in every compact subset of the half-plane. 
We also note that the Riemann zeta function 
can be extended to a meromorphic function on the 
complex plane $\mathbb{C}$ having a single pole 
at $s=1$ by analytic continuation. 
See e.g., Apostol \cite{Apo} 
for more details on the Riemann zeta function.

The Riemann zeta distribution
is defined as a probability distribution on $\R$
generated by $\zt(s)$.

\begin{df}[Riemann zeta distribution]\label{zetad}
Fix $\sg > 1$. 
A probability distribution $\mu_\sg$ on $\R$
is called a Riemann zeta distribution if
    \begin{equation}
    \label{zdist}
    \mu_\sg(\{- \log m \}) 
    = \frac{m^{-\sg}}{\zt(\sg)}, \qquad m \in \N.
    \end{equation}
\end{df}

Let $\sigma>1$. 
As is seen later in Section \ref{Riemann zeta dist}, 
it is verified that the normalized function 
    \[
    f_\sigma(t):=\frac{\zeta(\sigma+\ii t)}{\zeta(\sigma)}, 
    \qquad t \in \R,
    \]
is a characteristic function 
of the Riemann zeta distribution $\mu_\sigma$. 
Moreover, we can also see that the $\mu_\sigma$ 
is infinitely divisible
(see Propositions \ref{CF zeta} 
and \ref{Prop:zeta-levy-khintchine}). 
Note that we often identify $\mu_\sigma$ with a 
countably supported function 
$\mu_\sigma=\mu_\sigma(x)$ given by 
    \[
    \mu_\sigma(x)=\sum_{m=1}^\infty \frac{m^{-\sigma}}{\zeta(\sigma)}
    \bm{1}_{\{-\log m\}}(x), \qquad x \in \Lambda_1, 
    \]
where $\Lambda_1=\{-\log m \, | \, m \in \mathbb{N}\}$. 
Then, for $n \in \N$, 
the $n$-th convolution power of the function $\mu_\sigma$
is recursively given by $\mu_\sigma^{*1}=\mu_\sigma$ and 
    \[
    \begin{aligned}
    \mu_\sigma^{*n}(x)&=\sum_{y \in \Lambda_{n-1}}
    \mu_\sigma^{*(n-1)}(x-y)\mu_\sigma(y)\\
    &=\sum_{m_1, m_2, \dots, m_n=1}^\infty
    \frac{(m_1m_2\cdots m_n)^{-\sigma}}{\zeta(\sigma)^n}
    \bm{1}_{\{-\log(m_1m_2\cdots m_n)\}}(x), \qquad 
    x \in \Lambda_n, \, n=2, 3, \dots, 
    \end{aligned}
    \]
where the countable support $\Lambda_n$ of $\mu_\sigma^{*n}$ is given by
    \[
    \begin{aligned}
    \Lambda_n&=\{x \in \R \, | \, y \in \Lambda_{n-1}, 
    \, x-y \in \Lambda_1\}\\
    &=\{-\log (m_1m_2 \cdots m_n) \, | \, 
    m_1, m_2, \dots, m_n\in \N\}, \qquad n=2, 3, \dots. 
    \end{aligned}
    \]

By applying the Taylor formula to the function 
$\Gamma_\sigma(t):=\log f_\sigma(t)$ 
on a neighborhood of $0$, 
we know that there exist
$\alpha_\sigma \in \R$ and $\beta_\sigma>0$ such that
    \[
    \Gamma_\sigma(t):= 
    \ii \alpha_\sigma t - \beta_\sigma t^2 
    + o(t^2)
    \]
as $t \to 0$, 
where the constants $\alpha_\sigma$ and $\beta_\sigma$
can be regarded as the expectation and the variance of the 
Riemann zeta distribution $\mu_\sigma$, respectively. 
The explicit representations of these constants will be given 
in Lemma \ref{Lem:alphabeta-rep1}.

The main results of the present paper consist of two claims. 
The first main result is the very local limit theorem 
for the $n$-th convolution power $\mu_\sigma^{*n}$ 
of the Riemann zeta distribution $\mu_\sigma$, 
which reveals the leading term of the asymptotic behavior of $\mu_\sigma^{*n}$ as $n \to \infty$.

\begin{tm}\label{Thm:Local limit theorem}
Let $\sigma>1$.  
Then, we have 
    \begin{equation}\label{asymptotic}
    \mu_\sigma^{*n}(x)
    =\frac{1}{\sqrt{n}}
    p_\sigma\Big(\frac{x-\alpha_\sigma n}{\sqrt{n}}\Big)
    +o\Big(\frac{1}{\sqrt{n}}\Big), 
    \qquad x \in \Lambda_n,
    \end{equation}
as $n \to \infty$, where the function
    \[
    p_\sigma(x)
    :=\frac{1}{2\pi}\int_{\R} \e^{-\ii xu} 
    \e^{-\beta_\sigma u^2} \, \dd u
    =\frac{1}{\sqrt{4\pi \beta_\sigma}}
    \exp\Big(-\frac{x^2}{4\beta_\sigma}\Big), 
    \qquad x \in \R,
    \]
is the heat kernel evaluated at time $\beta_\sigma$. 
\end{tm}

The second main result of the present paper is 
the upper  bound of the supremum norm 
of $\mu_\sigma^{*n}$, which is stated as follows: 

\begin{tm}\label{Thm:lower and upper bound}
Let $\sigma>1$. 
Then, there exists a positive constant $C_\sigma>0$
such that 
    \begin{equation}
     \|\mu_\sigma^{*n}\|_\infty
     :=\sup_{x \in \Lambda_n}\mu_\sigma^{*n}(x)
    \le \frac{C_\sigma}{\sqrt{n}}, 
    \qquad n \in \N. 
    \end{equation}
\end{tm}

Let $\{X_\sigma^{(n)}\}_{n=1}^\infty$
be a sequence of i.i.d.~random variables whose 
common law is $\mu_\sigma$ with $\sigma>1$. 
Then, the usual central limit theorem implies that
the random variable defined by 
    \[
    \frac{X_\sigma^{(1)}+X_\sigma^{(2)}+
    \cdots + X_\sigma^{(n)}-n\alpha_\sigma}{\sqrt{n}},
    \qquad n \in \N,
    \]
converges in law to the normal distribution
$N(0, 2\beta_\sigma)$ as $n \to \infty$. Namely, we obtain 
    \[
    \lim_{n \to \infty}
    \mathbf{P}\Big(a \le 
    \frac{X_\sigma^{(1)}+X_\sigma^{(2)}+
    \cdots + X_\sigma^{(n)}-n\alpha_\sigma}{\sqrt{n}}
    \le b \Big)
    =\int_a^b \frac{1}{\sqrt{4\pi \beta_\sigma}}
    \exp\Big(-\frac{x^2}{4\beta_\sigma}\Big)\, \dd x 
    \]
for all $-\infty \le a < b \le \infty$. 
Our local limit theorem 
(Theorem \ref{Thm:Local limit theorem}) 
can be regarded as a refinement of the central limit theorem
for the sequence of Riemann zeta random variables 
$\{X_\sigma^{(n)}\}_{n=1}^\infty$. 
Moreover, it turns out from 
Theorem \ref{Thm:lower and upper bound}
that the supremum norm of $\mu_\sigma^{*n}(x)$
decays on at most the order of $\sqrt{n}$. 
We claim that our results are to be a specific case of known rather 
unknown examples of local limit theorems usually discussed in probability theory.

The rest of the present paper is organized as follows: 
We review some basics of 
the Riemann zeta distribution $\mu_\sigma, \, \sigma>1,$ in Section 2. 
In particular, the expectation and the variance 
of the Riemann zeta random variable are computed. 
We establish local behaviors of the characteristic function $f_\sigma(t)$
of $\mu_\sigma$ on a neighborhood of $0$ in Section 3. 
More precisely, we establish both the lower and the upper estimates 
of $|f_\sigma(t)|$ on the neighborhood 
by making use of the 
L\'evy--Khintchine representation of $f_\sigma(t)$
(see Lemmas \ref{Lem:lower estimate} and \ref{Lem:upper estimate}). 
Section 4 is devoted to the proof of the local limit theorem 
for $\mu_\sigma^{*n}$ as $n \to \infty$ 
(see Theorem \ref{Thm:Local limit theorem}). 
We employ the standard Fourier analysis technique 
and the upper estimate of $f_\sigma(t)$ in order to find out
the leading term of $\mu_\sigma^{*n}$ as $n$ tends to infinity. 
We also give the proof of Theorem \ref{Thm:lower and upper bound}, which provides the upper bound of the supremum norm of 
$\mu_\sigma^{*n}$, in Section 5. 
We give some further comments towards 
possible extensions of our study in Section 6 as well.



\section{{\bf The Riemann zeta distribution}}
\label{Riemann zeta dist}

Let $\sigma>1$ and consider 
the Riemann zeta distribution $\mu_\sigma$
defined by \eqref{zdist}. 
In this section, several properties of the
Riemann zeta distribution are exhibited. 
We refer to Lin--Hu \cite{LH} for related topics. 
The following fact is well-known. 

\begin{pr}[see e.g., Gnedenko--Kolmogorov \cite{GK}]
\label{CF zeta}
The characteristic function 
$f_\sg(t)$ of \eqref{zdist} is given by 
    \[
    f_\sg(t) := \frac{\zt(\sg + {\rm i}t)}{\zt(\sg)}, 
    \qquad t \in \R. 
    \]
\end{pr}
\noindent
We note that the function $f_{\sigma}(t)$ is also obtained by 
the Fourier transform of the function $\mu_\sigma(x)$. 
The Riemann zeta distribution is known to be  infinitely divisible. 

\begin{pr}[see e.g., Gnedenko--Kolmogorov \cite{GK}]
\label{Prop:zeta-levy-khintchine}
Let $\mu_{\sigma}$ be a Riemann zeta distribution on $\mathbb{R}$ 
with the characteristic function $f_{\sigma}(t)$.
Then, $\mu_{\sigma}$ is compound Poisson on 
$\mathbb{R}$ and it holds that
    \begin{align}
    \log f_{\sigma}(t)=
    \sum_{p \in \mathbb{P}}
    \sum_{r=1}^{\infty}
    \frac{p^{-r\sigma}}{r}\left(\e^{-{\rm i}tr\log p}-1\right)
    =\int_0^{\infty}\left(\e^{-{\rm i}tx}-1\right)N_{\sigma}(\dd x), \qquad 
    t \in \mathbb{R},
    \label{zeta CF}
    \end{align}
where $N_{\sigma}(\dd x)$ is a finite L\'evy measure on $\mathbb{R}$ given by
    \begin{align}\label{Levy-measure}
    N_{\sigma}(\dd x)
    =\sum_{p \in \mathbb{P}}
    \sum_{r=1}^{\infty}\frac{p^{-r\sigma}}{r}
    \delta_{r\log p}(\dd x).
    \end{align}
Here, $\delta_a(\dd x)$ denotes the delta measure 
at $a \in \R$. 
\end{pr}

By direct calculations, 
the first and the second derivatives of $\zt(s)$ are given by  
    \begin{align}
    \zeta'(s) 
    = \sum^{\infty}_{m=1} \frac{ -\log m}{m^s}, \qquad 
    \zeta''(s)
    = \sum^{\infty}_{m=1} \frac{(\log m)^2}{m^s},
    \label{2nd}
    \end{align}
respectively. 
Let $X_\sigma$ be a random variable whose distribution is  $\mu_\sigma$. 
Then, the expectations $\mathbb{E}[X_\sigma]$, $\mathbb{E}[X_\sigma^2]$ 
and the variance $\mathrm{Var}(X_\sigma)$ 
of $X_\sigma$ are given in the following. 

\begin{pr}\label{Prop:exp-var}
It holds that 
    \[
    \E[X_\sigma] =\frac{\zt '(\sg)}{\zt(\sg)} , \qquad 
    \E[X_\sigma^2]=\frac{\zt''(\sg)}{\zt(\sg)}, \qquad
    \mathrm{Var}(X_\sigma)=\frac{1}{\zt(\sg)^2}  
    \{\zeta(\sigma)\zt''(\sg) - \zt'(\sg)^2\}.
    \]
\end{pr}

\begin{proof}
It follows from  \eqref{2nd} that
    \begin{align*}
    \E[X_\sigma] 
    &= \sum_{m=1}^\infty (- \log m) \frac{m^{-\sg}}{\zt(\sg)} 
    =  \sum_{m=1}^\infty \frac{(- \log m)}{m^\sg}   \frac{1}{\zt(\sg)}
    = \frac{\zt '(\sg)}{\zt(\sg)}, \\
    \E[X_\sigma^2] &= \sum_{m=1}^\infty (- \log m)^2   \frac{m^{-\sg}}{\zt(\sg)} 
    =\sum_{m=1}^\infty \frac{(\log m)^2}{m^\sg}\frac{1}{\zt(\sg)}
    = \frac{\zt''(\sg)}{\zt(\sg)}, \\
     \mathrm{Var}(X_\sigma) &= \E[X_\sigma^2] - (\E[X_\sigma])^2\\
    &= \frac{\zt''(\sg)}{\zt(\sg)} 
    - \Big(\frac{\zt '(\sg)}{\zt(\sg)}\Big)^2
    = \frac{1}{\zt(\sg)^2}  
    \{\zeta(\sigma)\zt''(\sg) - \zt'(\sg)^2\},
    \end{align*}
which are the desired equalities. 
\end{proof}



\section{{\bf Local behaviors of the Riemann zeta distribution}}

Let us fix $\sigma>1$ and consider 
the  characteristic function $f_\sigma : \R \to \C$ 
of the Riemann zeta distribution $\mu_\sigma$
given in Proposition \ref{CF zeta}. 
By definition, we have $f_\sigma(0)=1$ 
and $|f_\sigma(t)| \le 1$
for $t \in \mathbb{R}$. 
In particular, $|f_\sigma(t)|$ attains its maximum 1 at $t=0$. 
Therefore, we are interested in some local behaviors of $f_\sigma(t)$ at $t=0$. 
In this section, we give both the lower and the upper bounds of $|f_\sigma(t)|$ on a neighborhood of $t=0$, 
which play a key role in the proof of main theorems. 
 
We put
$\Gamma_\sigma(t) \equiv \log f_{\sigma}(t), 
\, t \in \mathbb{R}$.
Since the function $\Gamma_\sg(t)$ is analytic on a 
neighborhood of $0$, 
we can consider its convergent Taylor series
    \[
    \Gm_\sigma(t) = \sum_{l = 1}^{\infty} a_l(\sigma) t^l
    \]
on a neighborhood of $0$.
We give the explicit representation of the coefficient 
$a_l(\sigma), \, l=1, 2, \dots$, 
by using Proposition \ref{Prop:zeta-levy-khintchine}. 

\begin{lm}\label{coef}
The coefficient $a_l(\sigma)$, $l=1, 2, \dots,$ is given by 
    \[
    a_l(\sigma) 
    = \sum_{p \in \Prime}
    \sum_{r = 1}^{\infty}
    \frac{p^{-r\sg}}{r} 
    \frac{(-{\rm i}r \log p)^l}{l!}, 
    \qquad l=1, 2, \dots.
    \]
\end{lm}

\begin{proof}
By applying  Proposition \ref{Prop:zeta-levy-khintchine}, we have
    \begin{align*}
    \Gm_\sigma(t) 
    &= \sum_{p \in \Prime}
    \sum ^{\infty} _{r = 1}
    \frac{p^{-r \sg}}{r} 
    (\e^{-{\rm i} tr\log p}-1) \\
    & = \sum_{p \in \Prime}
    \sum_{r = 1}^{\infty}
    \frac{p^{-r\sg}}{r} 
    \Big( \sum _{l=1} ^\infty 
    \frac{(-{\rm i}tr \log p)^l}{l!}\Big)
    = \sum _{l=1} ^\infty 
    \Big(\sum_{p \in \Prime}\sum_{r = 1}^{\infty}
    \frac{p^{-r\sg}}{r} 
    \frac{(-{\rm i}r \log p)^l}{l!} \Big) t^l
    \end{align*}
for every $t$ in a neighborhood of $0$. 
\end{proof}

The following proposition gives 
the lower estimate of the function 
$|f_\sigma(t)|$ on a neighborhood of $0$. 

\begin{lm}\label{Lem:lower estimate}
There exists a positive constant $C_\sigma>0$ such that 
$\exp(-C_\sigma t^2) \leq |f_\sg(t)|$
on a neighborhood of $0$. 
\end{lm}

\begin{proof}
It follows from Lemma \ref{coef} that 
    \begin{align*}
    |f_\sg(t)|
    &= \Big|\exp \Big ( \sum _{l=1} ^\infty \sum_{p \in \Prime}
    \sum_{r = 1}^{\infty}\frac{p^{-r\sg}}{r} \frac{(-{\rm i}r \log p)^l}{l!}  t^l \Big) \Big|\\
    &= \exp \Big ( \sum _{j=1} ^\infty \sum_{p \in \Prime}
    \sum_{r = 1}^{\infty}\frac{p^{-r\sg}}{r} 
    \frac{(r \log p)^{4j}}{(4j)!}  t^{4j} 
    - \sum _{j=1} ^\infty \sum_{p \in \Prime}
    \sum_{r = 1}^{\infty}\frac{p^{-r\sg}}{r} 
    \frac{(r \log p)^{4j-2}}{(4j-2)!} t^{4j-2}\Big) \\
    &\geq \exp\Big (- \sum _{j=1} ^\infty \sum_{p \in \Prime}
    \sum_{r = 1}^{\infty}\frac{p^{-r\sg}}{r} \frac{(r \log p)^{4j-2}}{(4j-2)!}  t^{4j-2}\Big) \\
    &\geq \exp \Big(- \sum _{j=1} ^\infty \sum_{p \in \Prime}
    \sum_{r = 1}^{\infty}\frac{p^{-r\sg}}{r} \frac{(r \log p)^{4j-2}}{(4j-2)!}  t^{2}\Big) = \exp(-C_\sigma t^2)
    \end{align*}
for every $t$ in a neighborhood of 0. 
\end{proof}
We put 
    \[
    \alpha_\sigma
    :=\sum_{p \in \mathbb{P}} 
    \frac{- \log p}{p^\sg -1}, \qquad 
    \beta_\sigma
    :=\sum_{p \in \Prime}
    \frac{ (\log p)^2}{2} \frac{p^\sg}{(p^\sg - 1)^2}.
    \]
Since it holds that 
$a_1(\sigma)=\ii \alpha_\sigma$ 
and $a_2(\sigma)=-\beta_\sigma$, 
we obtain 
    \begin{equation}\label{Gamma}
     \Gm_\sigma(t) 
     = {\rm i}\al_\sigma t - \bt_\sigma t^2 
     + \sum^\infty_{l =3} a_l(\sigma) t^l 
     \end{equation}
on a neighborhood of $0$. 
Moreover, it follows from 
Proposition \ref{Prop:zeta-levy-khintchine} that 
    \begin{align}
    \Gamma_\sg(t) 
    =\int^{\infty}_{0} \sum_{l=1}^\infty 
    \frac{(-{\rm i}tx)^l}{l!}\,N_\sg(\dd x) 
    = \sum_{l=1}^\infty  
    \Big(\int^{\infty}_{0} 
    \frac{(-{\rm i}x)^l}{l!}\,N_\sg(\dd x) \Big)  t^l  
    \label{Gamma2}
    \end{align}
on a neighborhood of 0. 
Therefore,  each coefficient $a_l(\sigma)$  is also represented  as
    \[
    a_l(\sigma) 
    = \int^{\infty}_{0} \frac{(-{\rm i}x)^l}{l!}N_\sg(dx), 
    \qquad l=1, 2, \dots.
    \]
Particularly, look at the cases where $l=1, 2$. 
It follows from \eqref{Levy-measure} that
    \[
    \begin{aligned}
    a_1(\sigma) 
    &=-\ii \int^{\infty}_{0} x 
    \sum_{p \in \Prime}\sum_{r = 1}^{\infty}
    \frac{p^{-r\sg}}{r} \dl_{r \log p} (\dd x)
    = -\ii\sum_{p \in \Prime}
    \sum_{r = 1}^{\infty} p^{-r\sg} \log p, \\
    a_2(\sigma) 
    &=-\frac{1}{2} \int^{\infty}_{0} x^2 
    \sum_{p \in \Prime}\sum_{r = 1}^{\infty}
    \frac{p^{-r\sg}}{r} \dl_{r \log p} (\dd x)
    = -\frac{1}{2}\sum_{p \in \Prime}\sum_{r = 1}^{\infty} 
    r p^{-r\sg} (\log p)^2.
    \end{aligned}
    \]
By putting it all together, we have the following. 
\begin{lm}\label{Lem:alphabeta-rep1}
It holds that
    \[
    \begin{aligned}
    \al_\sigma 
    &=-\int^{\infty}_{0} x \, N_\sg(\dd x)
    =\sum_{p \in \Prime} \frac{- \log p}{p^\sg -1}
    =-\sum_{p \in \Prime}\sum_{r = 1}^{\infty} p^{-r\sg} \log p, \\
    \beta_\sigma &=\frac{1}{2} \int^{\infty}_{0} x^2 \, N_\sg(\dd x)
    =\sum_{p \in \Prime}
    \frac{ (\log p)^2}{2} \cdot \frac{p^\sg}{(p^\sg - 1)^2}
    =\frac{1}{2}\sum_{p \in \Prime}\sum_{r = 1}^{\infty} 
    r p^{-r\sg} (\log p)^2.
    \end{aligned}
    \]
\end{lm}

Here, we also give another expressions 
of the constants $\al_\sigma$ and $\bt_\sigma$
in terms of the Riemann zeta random variable. 
Let $X_\sigma$ be a random variable 
whose distribution is $\mu_\sigma$. 
Then, we can show that 
\begin{equation}\label{alphabeta}
    \al_\sigma = \E[X_\sigma]
    =\frac{\zeta'(\sigma)}{\zeta(\sigma)}, 
    \qquad \bt_\sigma = \frac{1}{2}\mathrm{Var}(X_\sigma)
    =\frac{1}{2\zt(\sg)^2}  
    \{\zeta(\sigma)\zt''(\sg) - \zt'(\sg)^2\}.
    \end{equation}
Indeed, we have
    \[
    \begin{aligned}
    \frac{\dd}{\dd t} f_\sg(t) 
    = \frac{{\rm i} \zt'(\sg + {\rm i}t)}{\zt(\sg)}, \qquad
    \frac{\dd^2}{\dd t^2} f_\sg(t) 
    = -\frac{ \zt''(\sg + {\rm i}t)}{\zt(\sg)}
    \end{aligned}
    \]
for $t \in \R$. 
By letting $t = 0$ and 
by using Proposition \ref{Prop:exp-var}, 
we obtain
    \begin{align}
    \frac{\dd}{\dd t} f_\sg(0) = {\rm i}\frac{\zt'(\sg)}{\zt(\sg)} = {\rm i} \E[X_\sigma], \qquad
    \frac{\dd^2}{\dd t^2} f_\sg(0) = - \frac{\zt''(\sg)}{\zt(\sg)} = -\E [X^2].
    \label{beta-rep-1}
    \end{align}
On the other hand, 
the Lebesgue convergence theorem and 
Proposition \ref{Prop:zeta-levy-khintchine} imply
    \begin{align*}
    \frac{\dd}{\dd t} f_\sg(t) 
    &= -\ii f_\sg(t) \Big( \int^{\infty}_{0} 
    x \e^{-{\rm i}tx} \, N_\sg(\dd x)\Big),
    \qquad t \in \R,\\
    \frac{\dd^2}{\dd t^2}f_\sg(t) 
    &= f_\sg(t)\Big\{-\Big(\int^{\infty}_{0} 
    x \e^{-{\rm i}tx} \,N_\sg(\dd x)\Big)^2
    -\int^{\infty}_{0} x^2 \e^{-{\rm i}tx} \, N_\sg(\dd x)\Big\}.
    \qquad t\in \R. 
    \end{align*}
By letting $t = 0$, we also have 
    \begin{align}
    \frac{\dd}{\dd t} f_\sg(0) 
    &= -{\rm i} \int^{\infty}_{0} x \, N_\sg(\dd x), 
    \label{alpha-rep-2} \\
    \frac{\dd^2}{\dd t^2} f_\sg(0) 
    &= - \Big(\int^{\infty}_{0} x \, N_\sg(\dd x) \Big)^2 
    - \int^{\infty}_{0} x^2 \, N_\sg(\dd x).
    \label{beta-rep-2}
    \end{align}
Therefore, by combining Lemma \ref{Lem:alphabeta-rep1} with 
\eqref{beta-rep-1},  \eqref{alpha-rep-2} 
and \eqref{beta-rep-2}, we obtain
$$
\begin{aligned}
\mathbb{E}[X_\sigma] =\alpha_\sigma, \qquad 
\mathbb{E}[X_\sigma^2] =\alpha_\sigma^2+2\beta_\sigma, \qquad 
\mathrm{Var}(X_\sigma)=\mathbb{E}[X_\sigma^2]-(\mathbb{E}[X_\sigma])^2
=2\beta_\sigma,
\end{aligned}
$$
which are the desired equalities \eqref{alphabeta}.

The following lemma gives the upper estimate 
of $|f_\sigma(t)|$ on a neighborhood of 0.

\begin{lm}\label{Lem:upper estimate}
There exists a positive constant $B_\sigma > 0 $  such that
    \[
    |\Gm_\sigma(t) - {\rm i}\al_\sigma t + \bt_\sigma t^2| 
    \leq B_\sigma|t|^3, 
    \]
and 
    \begin{equation}
    \label{upper-bound}
     |f_\sg (t)| \leq \exp\Big(-\frac{1}{2}\beta_\sigma t^2\Big) 
    \end{equation}
hold on a neighborhood of $0$.   
\end{lm}

\begin{proof}
On a neighborhood of 0, we have
    \begin{align*}
     |\Gm_\sg(t) - {\rm i}\al_\sigma t + \bt_\sigma t^2| 
    &= \Big|\sum_{l = 3}^ \infty a_l(\sigma) t^l\Big| 
    = \Big|\sum_{l = 3}^ \infty \sum _{p \in \Prime} 
    \sum_{r= 1}^\infty \frac{p^{-r\sg}}{r} 
    \frac{(-{\rm i}r \log p)^l}{l!}   t^l \Big | \\
    &\leq  \sum_{l = 2}^ \infty \sum _{p \in \Prime} \sum_{r= 1}^\infty 
    \frac{p^{-r\sg}}{r} \frac{(r \log p)^l}{l!}  |t|^3 
    = B_\sigma|t|^3.
    \end{align*}
Moreover, for every $t$ in a neighborhood of 0, we have 
    \begin{align*}
    |f_\sg(t) | 
    &= \Big|\exp \Big( {\rm i} \al_\sigma t -\bt_\sg t^2 
    + \sum_{l = 3}^\infty a_l(\sigma) t^l \Big)\Big|\\
    &\le \exp \Big( -\bt_\sg t^2 
    + \Big| \sum_{l = 3}^\infty a_l(\sg) t^l \Big| \Big) 
    \le \exp\Big((- \bt_\sg + B_\sg|t|)t^2\Big).
    \end{align*}
By taking the neighborhood of 0 being sufficiently small, 
we have 
    \[
    - \bt_\sg + B_\sg|t| \leq -\frac{1}{2}\beta_\sg,
    \]
which concludes \eqref{upper-bound}  
on the neighborhood of 0. 
\end{proof}


\section{{\bf Proof of Theorem \ref{Thm:Local limit theorem}}}

The aim of this section is to give a proof of the local limit theorem 
for the convolution power of the Riemann zeta distribution $\mu_\sigma$ with 
fixed $\sigma>1$ (Theorem \ref{Thm:Local limit theorem}). 
The following function will play a crucial role in the proof. 
We define a function $p_\sigma : \mathbb{R} \to \mathbb{R}$ by 
    \[
    p_\sigma(x)
    :=\frac{1}{2\pi}\int_{\R} 
    \e^{-\ii xu} \e^{-\beta_\sigma u^2} \, \dd u
    =\frac{1}{\sqrt{4\pi \beta_\sigma}}
    \exp\Big(-\frac{x^2}{4\beta_\sigma}\Big), 
    \qquad x \in \R,
    \]
which is referred to the heat kernel on $\R$ 
evaluated at time $\beta_\sigma>0$. 
In probability theory, 
it is well-known that the convolution powers 
of probability distributions 
are often approximated by the heat kernel. 
In this section, we also see that such local limit theorems
are valid for the case of Riemann zeta distributions.
For this sake, we first show the following lemma. 

\begin{lm}\label{Lem:local-limit-1}
For any $\ve>0$, there exist $\delta>0$ and $N \in \N$ such that 
    \begin{equation}
    \label{Local-limit-1}
    \Big|\frac{\sqrt{n}}{2\pi}\int_{|t|<\delta} 
    f_\sigma(t)^n \e^{-\ii x t} \, \dd t 
    - p_\sigma\Big(\frac{x-\alpha_\sigma n}{\sqrt{n}}\Big)\Big|
    <\ve
    \end{equation}
for all $n \ge N$ and $x \in \R$. 
\end{lm}

\begin{proof}
Since the function $u \mapsto \exp(-\beta_\sigma u^2/2)$ is 
Lebesgue integrable on $\R$, 
for any $\ve>0$, there exists $M>0$ such that 
    \begin{equation}\label{integrable}
    \int_{|u| \ge M} 
    \exp\Big(-\frac{1}{2}\beta_\sigma u^2\Big) \, \dd u 
    <\frac{2\pi \ve}{3}. 
    \end{equation}
For such $M>0$, we take a sufficiently large $n \in \N$
with $M<\delta \sqrt{n}$.  
By using Lemma \ref{Lem:upper estimate}, 
there exists $\delta>0$ such that 
$|u|<\delta \sqrt{n}$ implies  
    \begin{align}
    \Big|\exp\Big(-\frac{\ii \alpha_\sigma u}{\sqrt{n}}\Big)f_\sigma\Big(\frac{u}{\sqrt{n}}\Big)\Big|^n
    &\le \exp\Big(-\frac{1}{2}\beta_\sigma \Big(\frac{u}{\sqrt{n}}\Big)^2\Big)^n
    =\exp\Big(-\frac{1}{2}\beta_\sigma u^2\Big). 
    \label{est}
    \end{align}
Since it holds that
    \[
    \begin{aligned}
    &\frac{\sqrt{n}}{2\pi}\int_{|t|<\delta} 
    f_\sigma(t)^n \e^{-\ii x t} \, \dd t \\
    &=\frac{1}{2\pi} \int_{|u|<\delta\sqrt{n}} 
    \Big\{\exp\Big(-\frac{\ii \alpha_\sigma u}{\sqrt{n}}\Big)
    f_\sigma\Big(\frac{u}{\sqrt{n}}\Big)\Big\}^n
    \exp\Big(-\ii u\frac{x- \alpha_\sigma n}{\sqrt{n}}\Big) \, \dd u, 
    \end{aligned}
    \]
we have 
    \[
    \begin{aligned}
    &\Big|\frac{\sqrt{n}}{2\pi}\int_{|t|<\delta} 
    f_\sigma(t)^n \e^{-\ii x t} \, \dd t 
    - p_\sigma\Big(\frac{x-\alpha_\sigma n}{\sqrt{n}}\Big)\Big| \\
    &=\Big|\frac{1}{2\pi} \int_{|u|<\delta\sqrt{n}} 
    \Big\{\exp\Big(-\frac{\ii \alpha_\sigma u}{\sqrt{n}}\Big)f_\sigma
    \Big(\frac{u}{\sqrt{n}}\Big)\Big\}^n
    \exp\Big(-\ii u\frac{x- \alpha_\sigma n}{\sqrt{n}}\Big) \, \dd u\\
    &\hspace{1cm}- \frac{1}{2\pi} \int_{\R} 
    \exp\Big(-\ii u\frac{x- \alpha_\sigma n}{\sqrt{n}}\Big)
     \e^{-\beta_\sigma u^2} \, \dd u \Big|\\
    &\le \frac{1}{2\pi} \int_{|u|<M} 
     \Big|\Big\{\exp\Big(-\frac{\ii \alpha_\sigma u}{\sqrt{n}}\Big)
     f_\sigma\Big(\frac{u}{\sqrt{n}}\Big)\Big\}^n
     -\e^{-\beta_\sigma u^2} \Big| \, \dd u \\
    &\hspace{1cm}+\frac{1}{2\pi} \int_{M \le |u| < \delta\sqrt{n}}
    \Big|\exp\Big(-\frac{\ii \alpha_\sigma u}{\sqrt{n}}\Big)
    f_\sigma\Big(\frac{u}{\sqrt{n}}\Big)\Big|^n \, \dd u 
    +\frac{1}{2\pi}\int_{|u| \ge M}\e^{-\beta_\sigma u^2} \, \dd u\\
    &=: \mathcal{I}_1+\mathcal{I}_2+\mathcal{I}_3,
    \end{aligned}
    \]
We first consider the term $\mathcal{I}_1$. 
Since 
    \[
    \begin{aligned}
    &\lim_{n \to \infty}
     \Big|\Big\{\exp\Big(-\frac{\ii \alpha_\sigma u}{\sqrt{n}}\Big)
     f_\sigma\Big(\frac{u}{\sqrt{n}}\Big)\Big\}^n
     -\e^{-\beta_\sigma u^2} \Big| \\
    &=\lim_{n \to \infty} 
    \e^{- \beta_\sigma u^2}
    \Big|\exp\Big(\frac{1}{\sqrt{n}}
    \sum_{l=3}^\infty a_l(\sigma)\frac{u^l}{n^{(l-3)/2}} \Big)
    -1\Big|=0, 
    \qquad |u| \le M,
    \end{aligned}
    \]
there exists $N \in \N$ such that $N>(M/\delta)^2$ implies $\mathcal{I}_1<\ve/3$ 
by applying the Lebesgue convergence theorem. 
As for the terms $\mathcal{I}_2$ and $\mathcal{I}_3$, \eqref{integrable} and \eqref{est} yield
    \[
    \begin{aligned}
    \mathcal{I}_2 
    &\le \frac{1}{2\pi} \int_{M \le |u| < \delta\sqrt{n}}
    \exp\Big(-\frac{1}{2}\beta_\sigma u^2\Big) \, \dd u 
    \le  \frac{1}{2\pi} \int_{|u| \ge M} \exp\Big(-\frac{1}{2}\beta_\sigma u^2\Big) \, \dd u 
    <\frac{\ve}{3}, \\
    \mathcal{I}_3 
    &\le \frac{1}{2\pi}\int_{|u| \ge M}
    \e^{-\beta_\sigma u^2} \, \dd u 
    \le  \frac{1}{2\pi} \int_{|u| \ge M} \exp\Big(-\frac{1}{2}\beta_\sigma u^2\Big) \, \dd u 
    <\frac{\ve}{3}. 
    \end{aligned}
    \]
By putting it all together, 
we have established \eqref{Local-limit-1}. 
\end{proof}

We are now ready for the proof of 
Theorem \ref{Thm:Local limit theorem}.

\begin{proof}[Proof of Theorem \ref{Thm:Local limit theorem}]
Since the characteristic function 
$f_\sigma(t)$ is obtained 
as the Fourier transform of 
the function $\mu_\sigma(x)$, we also obtain 
    \begin{equation}\label{inverse-Fourier}
    \mu_\sigma^{*n}(x)
    =\frac{1}{2\pi}\int_{-\pi}^\pi 
    f_\sigma(t)^n \e^{-\ii xt} \, \dd t, \qquad 
    n \in \N, \, x \in \Lambda_n.
    \end{equation}
Let $\ve>0$. 
By taking a sufficiently small $\delta>0$, we have 
    \begin{align}
    \sqrt{n}\mu_\sigma^{*n}(x) 
    &= \frac{\sqrt{n}}{2\pi} \int_{|t|<\pi}
    f_\sigma(t)^n \e^{-\ii xt} \, \dd t\nonumber  \\
    &=\frac{\sqrt{n}}{2\pi} \int_{|t| < \delta} 
    f_\sigma(t)^n \e^{-\ii xt} \, \dd t
    +\frac{\sqrt{n}}{2\pi} \int_{\delta \le |t| <\pi} 
    f_\sigma(t)^n \e^{-\ii xt} \, \dd t \nonumber \\
    &=: \mathcal{J}_1+\mathcal{J}_2 \nn
    \label{eq:LL1}
    \end{align}
for all $x \in \Lambda_n$.
Let us consider the term $\mathcal{J}_1$. By virtue of 
Lemma \ref{Lem:local-limit-1}, 
there exist $\delta>0$ and $N \in \N$ such that 
    \begin{equation*}
    \Big|\frac{\sqrt{n}}{2\pi}\int_{|t|<\delta} 
    f_\sigma(t)^n \e^{-\ii xt} \, \dd t 
    - p_\sigma\Big(\frac{x-\alpha_\sigma n}{\sqrt{n}}\Big)\Big|
    <\frac{\ve}{2}. 
    \end{equation*}
for $n \ge N$ and $x \in \Lambda_n$. 
Hence, we have 
    \[
    |\mathcal{J}_1| \le p_\sigma
    \Big(\frac{x-\alpha_\sigma n}{\sqrt{n}}\Big)
    +\frac{\ve}{2}, \qquad x \in \Lambda_n.
    \]
On the other hand, as for the term $\mathcal{J}_2$, we also have
    \[
    |\mathcal{J}_2| \le 
    \frac{\sqrt{n}}{2\pi}\int_{\delta \le |t|<\pi} 
    |f_\sigma(t)|^n \, \dd t
    \le \sqrt{n}\Big(\sup_{\delta \le |t|<\pi}|f_\sigma(t)|\Big)^n. 
    \]
By noting the fact that $\sup_{\delta \le |t|<\pi}|f_\sigma(t)|<1$, 
we take a sufficiently large $N' \in \N$ such that 
$n \ge N'$ implies that
$|\mathcal{J}_2|<\ve/2$. Therefore, we obtain 
    \[
    \Big|\sqrt{n}\mu_\sigma^{*n}(x) -p_\sigma\Big(\frac{x-\alpha_\sigma n}{\sqrt{n}}\Big)\Big|
    \le \ve
    \]
for all $n \ge \max\{N, N'\}$ and $x \in \Lambda_n$, 
which leads to 
    \[
    \mu_\sigma^{*n}(x)
    =\frac{1}{\sqrt{n}}p_\sigma
    \Big(\frac{x-\alpha_\sigma n}{\sqrt{n}}\Big)
    +o\Big(\frac{1}{\sqrt{n}}\Big)
    \]
as $n \to \infty$ for all $x \in \Lambda_n$. 
\end{proof}

\section{{\bf Proof of Theorem 
\ref{Thm:lower and upper bound}
}}

In order to show Theorem 
\ref{Thm:lower and upper bound}, 
we need to give the proof of the following lemma.

\begin{lm}\label{Lem:pre}
There exists a sufficiently small $\delta>0$ such that 
$$
\Big|\frac{1}{2\pi} \int_{|t|<\delta} f_\sigma(t)^n \e^{-\ii xt} \, \dd t\Big| \le \frac{1}{\sqrt{n\beta_\sigma}}, 
\qquad n \in \N, \, x \in \Lambda_n. 
$$
\end{lm}

\begin{proof}
By applying Lemma \ref{Lem:upper estimate}, there exists a sufficiently small 
$\delta>0$ such that 
$$
|f_\sigma(t)| \le \exp\Big(-\frac{1}{2}\beta_\sigma t^2\Big), \qquad |t|<\delta. 
$$
Therefore, we obtain 
    \[
    \begin{aligned}
    \Big|\frac{1}{2\pi} \int_{|t|<\delta} 
    f_\sigma(t)^n \e^{-\ii xt} \, \dd t \Big|
    &\le \frac{1}{2\pi} \int_{|t|<\delta} 
    |f_\sigma(t)|^n \, \dd t \\
    &\le \frac{1}{2\pi} \int_{\R}
    \exp\Big(-\frac{n}{2}\beta_\sigma t^2\Big) \, \dd t 
    =\frac{1}{\sqrt{n\beta_\sigma}}
    \end{aligned}
    \]
for $n \in \N$ and $x \in \Lambda_n$. 
\end{proof}

At last, we give the proof of Theorem \ref{Thm:lower and upper bound}.

\begin{proof}
[Proof of Theorem \ref{Thm:lower and upper bound}]

By virtue of \eqref{inverse-Fourier} 
and Lemma \ref{Lem:upper estimate}, there is a sufficiently small $\delta>0$
such that 
    \[
    \begin{aligned}
    \mu_\sigma^{*n}(x) 
    &\le \Big|\frac{1}{2\pi}\int_{|t|<\delta} 
    f_\sigma(t)^n \e^{-\ii xt} \, \dd t\Big| + \Big|\frac{1}{2\pi}\int_{\delta \le |t|<\pi} 
    f_\sigma(t)^n \e^{-\ii xt} \, \dd t\Big|\\
    &\le \frac{1}{\sqrt{n\beta_\sigma}} 
    + \Big(\sup_{\delta \le |t| <\pi} 
    |f_\sigma(t)|\Big)^n, \qquad 
    n \in \N, \, x \in \Lambda_n.
    \end{aligned}
    \]
By noting $\sup_{\delta \le |t| <\pi} |f_\sigma(t)| <1$
and by taking the supremum over $x \in \Lambda_n$, there exists a sufficiently large $C_\sigma>0$ such that  
    \[
    \|\mu_\sigma^{*n}\|_\infty 
    \le \frac{C_\sigma}{\sqrt{n}},  
    \qquad n \in \N,
    \]
which is the desired upper bound.
\end{proof}

\section{{\bf Conclusion}}

Throughout the present paper, 
we have discussed asymptotic behaviors 
including local limit theorems
for $n$-th convolution powers of Riemann zeta distributions
as $n \to \infty$. 
We believe that our result exactly 
contributes to the study of 
convolution powers of complex functions
in that the function treated in the present paper 
has a not finite but countable support. 
On the other hand, 
we can find several studies in which 
relations between generalizations of 
the Riemann zeta function and 
probability distributions on $\R^d$
are discussed. 
We refer to Hu--Iksanov--Lin--Zakusylo \cite{HILZ} for 
probability distributions generated by 
the Hurwitz zeta function
and Aoyama--Nakamura \cite{AN1, AN2} for those 
generated by certain classes of 
multidimensional zeta functions.
Such zeta functions are also 
countably supported and 
have large amount of applications 
in both probability theory and number theory. 
Hence, further studies for investigating 
precise asymptotics of convolution powers 
of zeta distributions 
are expected, which should be interesting problems 
to reveal.

\vspace{3mm}
\noindent
{\bf Acknowledgement.} 
The second-named author is supported by JSPS KAKENHI
Grant number 19K23410.



\end{document}